\documentclass[12pt]{amsart}

\theoremstyle{plain}
\newtheorem{theorem}{Theorem}[section]

\newtheorem{corollary}[theorem]{Corollary}

\newtheorem{lemma}[theorem]{Lemma}

\newtheorem{proposition}[theorem]{Proposition}

\newtheorem{problem}[theorem]{Problem}

\theoremstyle{definition}
\newtheorem{example}[theorem]{Example}

\newtheorem{construction}[theorem]{Construction}

\numberwithin{equation}{section}

\usepackage{fullpage}
\usepackage{amssymb}
\usepackage{tikz}

\usepackage{colortbl}

\def\chh{\cellcolor{blue!30}}
\def\cll{\cellcolor{blue!15}}
\def\chl{\cellcolor{red!30}}
\def\clh{\cellcolor{red!15}}

\def\setof#1#2{\{#1:\;#2\}}
\def\d{d}

\begin{document}

\title{Permutations in two dimensions that maximally separate neighbors}

\author{Mohammed Albow}

\author{Jeff Edgington}

\author{Mario Lopez}

\address[Albow, Edgington, Lopez]{Department of Computer Science, University of Denver, 2155 East Wesley Avenue, Denver, Colorado 80208, U.S.A.}

\email[Albow]{Mohammed.Al-Bow@du.edu}

\email[Edgington]{Jeffrey.Edgington@du.edu}

\email[Lopez]{mlopez@du.edu}

\author{Petr Vojt\v{e}chovsk\'y}

\address[Vojt\v{e}chovsk\'y]{Department of Mathematics, University of Denver, 2390 S York Street, Denver, Colorado 80208, U.S.A.}

\email[Vojt\v{e}chovsk\'y]{petr@math.du.edu}

\begin{abstract}
We characterize all permutations on even-by-even grids that maximally separate neighboring vertices. More precisely, let $n_1$, $n_2$ be positive even integers, let $I(n_1,n_2)=\{1,\dots,n_1\}\times\{1,\dots,n_2\}$ be the $n_1\times n_2$ grid, let $\d$ be the $L_1$ metric on $I(n_1,n_2)$, and let $N=\setof{\{x,y\}\in I(n_1,n_2)\times I(n_1,n_2)}{\d(x,y)=1}$ be the set of neighbors in $I(n_1,n_2)$. We characterize all permutations $\pi$ of $I(n_1,n_2)$ that maximize $\sum_{\{x,y\}\in N} \d(\pi(x),\pi(y))$.
\end{abstract}


\thanks{Research partially supported by the Simons Foundation Collaboration Grant 210176 to Petr Vojt\v{e}chovsk\'y}

\maketitle

\section{Introduction}

The \emph{total displacement} $\sum_{i=1}^n |\pi(i)-i|$ of a permutation $\pi$ on $\{1,2,\dots,n\}$ was investigated by Diaconis and Graham \cite{DG}, Gallero \emph{et al.} \cite{GEtal}, Daly and Vojt\v{e}chovsk\'y \cite{DV}, Guay-Paquet and Petersen \cite{GP}, B\"artschi \emph{et al.} \cite{B}, and others. The concept appears in the literature under various names and with subtle variations; the above definition and terminology is due to Knuth \cite[Problem 5.1.1.28]{Knuth}. The displacements $|\pi(i)-i|$ play a role in the construction of interleavers for turbo codes \cite{BEtal}.

\medskip

In this paper we study a quantity similar to total displacement that measures how much permutations \emph{separate} neighboring points. Let $X$ be a finite set, $\d$ a metric on $X$, and $N=\setof{\{x,y\}\in X\times X}{\d(x,y)=1}$ the set of (unit) neighbors in $X$. For a permutation $\pi$ of $X$, the quantity
\begin{displaymath}
    \frac{1}{|N|}\sum_{\{x,y\}\in N} \d(\pi(x),\pi(y))
\end{displaymath}
yields the average distance between $\pi$-images of neighboring points. Without the normalizing factor $1/|N|$, we obtain the function
\begin{equation}\label{Eq:f}
    f:S_{(X,d)}\to\mathbb N,\quad f(\pi) = \sum_{\{x,y\}\in N} \d(\pi(x),\pi(y)).
\end{equation}
We are interested in the properties of $f$, in particular in the following questions: \emph{What is the maximal value of $f$? For which permutations $\pi$ of $X$ is the maximum of $f$ attained?}

\medskip

These and other questions were studied and answered in \cite{DV} for the simplest case of a one-dimensional interval $X=\{1,2,\dots,n\}$ equipped with the $L_1$ metric $\d(i,j)=|i-j|$, cf. \cite[Theorem 3.1]{DV} or Theorem \ref{Th:DV}. Here we answer the two questions for permutations on even-by-even grids $I(2t_1,2t_2)$ equipped with the two-dimensional $L_1$ metric. The solution is much more intricate than in one dimension and can be roughly described as follows:
\begin{itemize}
\item the four entries at the center of the grid must be moved into the four corners,
\item further concentric layers surrounding the center must be moved onto the boundary of the grid,
\item images must maintain a double checkerboard pattern as much as possible, in the sense that if $x=(x_1,x_2)$, $y=(y_1,y_2)$ are the images of neighbors, then the values $x_1$, $y_1$ must not be both small (in $[1,\dots,t_1])$ nor both large, and similarly for $x_2$, $y_2$,
\item the unavoidable defects in the double checkerboard pattern must be located along the line parallel to the shorter dimension that partitions the grid into two equal halves,
\item the values involved in the defects must alternate between $t_1$ and $t_1+1$, or between $t_2$ and $t_2+1$.
\end{itemize}
See Construction \ref{Co:Pi} for a more precise description of the solutions, Figure \ref{Fg:OptimalSigma} for an example on the $6\times 6$ grid, Theorem \ref{Th:Main} for a complete description of the solutions in the generic case, and the comments following Theorem \ref{Th:Main} for the four exceptional cases.

\medskip

The paper concludes with comments on the general two-dimensional case and on higher-dimensional cases, where the solutions are suspected to behave quite differently.

\medskip

A playful example that motivates the two-dimensional problem is as follows. \emph{Students are about to take a test. Assuming that the students occupy all classroom seats arranged in a rectangular grid, how can the students be reseated so that the average distance to their previous neighbors is maximized?} Every experienced instructor will immediately recognize the practical importance of this problem.

We are not aware of an engineering application of our result. Nevertheless, two-dimensional interleavers have been considered for the transmission of two-dimensional data in noisy channels \cite{HEtal}, as were two-dimensional constructions of regular interleavers for turbo codes \cite{BEtal}.

\subsection{Notation}

All sets in this paper are treated as multisets and they will be denoted by uppercase letters. If $A$ is a multiset then $|A|$ denotes its cardinality and $\sum A$ is a shorthand for $\sum_{x\in A} x$. The notation $k*x$ indicates that there are $k$ copies of $x$, for instance, $\{3*2,2*4\} = \{2,2,2,4,4\}$.

We denote the coordinates of $x = (i,j)$ of $\mathbb Z\times \mathbb Z$ by $x_1=i$ and $x_2=j$. The subscript $1$ refers to rows, while the subscript $2$ refers to columns. If $A$ is a submultiset of $\mathbb Z\times\mathbb Z$ and $i\in\{1,2\}$, we define
\begin{displaymath}
    A_i = \setof{x_i}{x\in A}.
\end{displaymath}


Fix positive integers $n_1$, $n_2$ and let $I=I(n_1,n_2) = \{1,\dots,n_1\}\times \{1,\dots,n_2\}$ be the $n_1\times n_2$ grid. We will visualize a permutation $\pi$ of $I$ as an $n_1\times n_2$ array with the coordinate $(1,1)$ in the top left corner (that is, we are using matrix coordinates), where we place $(k,\ell)$ in row $i$ and column $j$ if and only if $\pi(i,j)=(k,\ell)$.

Given a permutation $\pi$ of $I$, let
\begin{displaymath}
    C=\{\pi(1,1),\pi(1,n_2),\pi(n_1,1),\pi(n_1,n_2)\}
\end{displaymath}
be the values of $\pi$ on the four corners of $I$, and let
\begin{displaymath}
    B=\setof{\pi(i,j)}{i\in\{1,n_1\}\text{ or } \,j\in\{1,n_2\}}\setminus C
\end{displaymath}
be the values of $\pi$ on the boundary of $I$ excluding the four corners.

Consider the $L_1$ metric $\d$ on $I$ defined by
\begin{displaymath}
    \d((i,j),(k,\ell)) = |i-k|+|j-\ell|.
\end{displaymath}
Note that the corners have two neighbors, other boundary cells have three neighbors, and interior cells have four neighbors in $I$.

As a special case of the general definition \eqref{Eq:f}, for a permutation $\pi$ of $I$ let
\begin{equation}\label{Eq:L1}
    \begin{array}{lll}
        f(\pi)  &=&\sum_{\substack{1\le i<n_1\\ 1\le j\le n_2}}\left(|\pi(i,j)_1-\pi(i+1,j)_1|+|\pi(i,j)_2-\pi(i+1,j)_2|\right)\\
                &+&\sum_{\substack{1\le i\le n_1\\ 1\le j<n_2}}\left(|\pi(i,j)_1-\pi(i,j+1)_1|+|\pi(i,j)_2-\pi(i,j+1)_2|\right).
    \end{array}
\end{equation}

Let $S_1 = \setof{4n_2*i}{1\le i\le n_1}$ and $S_2=\setof{4n_1*i}{1\le i\le n_2}$. The entries $\pi(i,j)_k$ of the combined sums of \eqref{Eq:L1} form a submultiset of $S_1\cup S_2$. Therefore, there are suitable submultisets $S_i^+ = S_i^+(\pi)$, $S_i^- = S_i^-(\pi)$ of $S_i$ such that
\begin{equation}\label{Eq:L1PN}
    f(\pi) = \sum S_1^+ - \sum S_1^- + \sum S_2^+ - \sum S_2^-.
\end{equation}
Note that $|S_i^+| = |S_i^-| = n_1(n_2-1)+n_2(n_1-1)$.

Here is a small example:

\begin{example}
Consider the permutation $\pi$ of $I(2,3)$ given by
\begin{displaymath}
    \begin{array}{|c|c|c|}
    \hline
    (2,2)&(1,3)&(2,1)\\
    \hline
    (1,2)&(1,1)&(2,3)\\
    \hline
    \end{array}
\end{displaymath}
Then
\begin{align*}
    f(\pi) &= (|2-1|+|2-3|) + (|1-2|+|3-1|) + (|1-1|+|2-1|) + (|1-2|+|1-3|)\\
        &\quad\quad+ (|2-1|+|2-2|) + (|1-1|+|3-1|) + (|2-2|+|1-3|) = 14.
\end{align*}
Upon separating the row and column coordinates, we get
\begin{align*}
    f(\pi) &= (|2-1|+|1-2|+|1-1|+|1-2|+|2-1|+|1-1|+|2-2|)\\
        &\quad\quad+(|2-3|+|3-1|+|2-1|+|1-3|+|2-2|+|3-1|+|1-3|).
\end{align*}
Finally, upon simplifying the absolute values and collecting positive and negative entries, we obtain
\begin{align*}
    f(\pi)  &= (2+2+1+2+2+1+2) - (1+1+1+1+1+1+2)\\
            &\quad\quad + (3+3+2+3+2+3+3) - (2+1+1+1+2+1+1)\\
            &=\sum S_1^+ - \sum S_1^- + \sum S_2^+ - \sum S_2^-,
\end{align*}
where $S_1^+ = \{2*1,5*2\}$, $S_1^- = \{6*1, 1*2\}$, $S_2^+=\{2*2,5*3\}$ and $S_2^-=\{5*1,2*2\}$ are the desired multisets.
\end{example}

\subsection{An unattainable upper bound}

From now on suppose that $n_1=2t_1$ and $n_2=2t_2$. For $A_i\subseteq S_i$, let
\begin{displaymath}
    A_i^< = \setof{u\in A_i}{1\le u\le t_i},\quad
    A_i^> = \setof{u\in A_i}{t_i<u\le n_i}.
\end{displaymath}
We will refer to the elements of $A_i^<$ as the \emph{small} values of $A_i$, and to the elements of $A_i^>$ as the \emph{large} values of $A_i$. We have $A_i^<\cup A_i^>=A_i$, of course.

As a special case we obtain the multisets $S_i^<$, $S_i^>$ with cardinalities $2n_1n_2$. Whenever two subsets $A_i^<\subseteq S_i^<$ and $A_i^>\subseteq S_i^>$ are given, we will assume that $A_i = A_i^<\cup A_i^>$ is defined.

\begin{lemma}\label{Lm:UpperBound}
Let $\pi$ be a permutation of $I(n_1,n_2)$ with $n_i=2t_i$. Then
\begin{displaymath}
    f(\pi) < \sum_{i=1}^2\left(\sum S_i^> - \sum S_i^<\right).
\end{displaymath}
\end{lemma}
\begin{proof}
The equality \eqref{Eq:L1PN} holds for suitable submultisets $S_i^+$, $S_i^-$ of $S_i$. We have $|S_i^+| = |S_i^-| = n_1(n_2-1)+n_2(n_1-1) < 2n_1n_2 = |S_i^>| = |S_i^<|$. Let $X_i^+$ be the multiset consisting of the $|S_i^+|$ largest values in $S_i^>$, and $X_i^-$ the multiset consisting of the $|S_i^-|$ smallest values in $S_i^<$. Then, for $i\in\{1,2\}$, we have
\begin{displaymath}
    \sum S_i^+ - \sum S_i^- \le \sum X_i^+ - \sum X_i^- < \sum S_i^> - \sum S_i^<.
\end{displaymath}
Summing up the inequalities for $i\in\{1,2\}$ finishes the proof.
\end{proof}

The upper bound of Lemma \ref{Lm:UpperBound} cannot be achieved for two reasons. First, due to the boundary of the grid, not every symbol from $S_i$ is used in $f(\pi)$ with the maximal multiplicity. Second, to count all large values with positive signs and all small values with negative signs in \eqref{Eq:L1PN}, the permutation $\pi$ would have to form a checkerboard with respect to the sets $S_1^<$, $S_1^>$ in the first coordinate, and another checkerboard with respect to the sets $S_2^<$, $S_2^>$ in the second coordinate. This is not possible. Indeed, once we place an entry from $S_1^>\times S_2^>$ in the array defining $\pi$, we must continue the double checkerboard pattern by alternating entries from $S_1^>\times S_2^>$ with entries from $S_1^<\times S_2^<$. Since an entry from $S_1^<\times S_2^>$, say, has to be used eventually, it will be adjacent to either an entry from $S_1^>\times S_2^>$ or to an entry from $S_1^<\times S_2^<$, and it will break the pattern in one of the coordinates.

\section{Corners, boundaries and defects in the double checkerboard pattern}

\subsection{Defects}

To investigate the flaws in the double checkerboard pattern, we define the \emph{defect} multisets
\begin{align*}
    D_i^< &= \setof{\max\{\pi(x)_i,\pi(y)_i\}}{\{x,y\}\in N,\,\{\pi(x)_i,\pi(y)_i\}\subseteq S_i^<},\\
    D_i^> &= \setof{\min\{\pi(x)_i,\pi(y)_i\}}{\{x,y\}\in N,\,\{\pi(x)_i,\pi(y)_i\}\subseteq S_i^>},
\end{align*}
for $i\in\{1,2\}$.

The definition of $D_1^<$ is justified as follows. If $x$, $y$ are neighboring cells and if the two values $u=\pi(x)_1$, $v=\pi(y)_1$ are both small, say with $u<v$, then these two values will contribute to $f(\pi)$ by $-u+v$, while in the ideal (and globally unattainable) situation we would like them to contribute by $-u-v$. We therefore keep track of the value with the incorrect sign, namely $v=\max\{u,v\}$. Similarly for the sets $D_1^>$, $D_2^<$, $D_2^>$.

\medskip

From now on, we will color cells with values in $S_1^>\times S_2^>$ dark blue, $S_1^<\times S_2^<$ light blue, $S_1^>\times S_2^<$ dark red, and $S_1^<\times S_2^>$ light red. We will call the (blue) values in $(S_1^>\times S_2^>)\cup(S_1^<\times S_2^<)$ \emph{homogeneous}, and the (red) values in $(S_1^>\times S_2^<)\cup (S_1^<\times S_2^>)$ \emph{heterogeneous}.

Let us give another example to elucidate many of the concepts defined so far.

\begin{example}\label{Ex:Defect}
Consider the permutation $\pi$ of $I(4,6)$ defined by
\begin{displaymath}
    \begin{array}{|c|c|c|c|c|c|}
        \hline
        \clh(2,6)&\chl(3,1)&\cll(1,1)&\chh(3,4)&\clh(1,4)&\chh(3,6)\\
        \hline
        \chl(4,3)&\clh(2,5)&\chl(3,2)&\cll(1,2)&\chh(3,5)&\cll(2,2)\\
        \hline
        \clh(2,4)&\chl(4,2)&\cll(1,3)&\chh(4,6)&\cll(2,1)&\chh(4,4)\\
        \hline
        \chl(4,1)&\clh(1,6)&\chl(3,3)&\clh(1,5)&\chh(4,5)&\cll(2,3)\\
        \hline
    \end{array}
\end{displaymath}
Since the values inside the homogeneous (resp. heterogeneous) region happen to form a double checkerboard in this example, the defects can occur only where the two regions meet. For instance, the adjacent entries $(1,4)$, $(3,6)$ in the first row contribute a column defect $\min\{4,6\}=4$, while the adjacent entries $(1,1)$, $(3,2)$ in the third column contribute a column defect $\max\{1,2\}=2$. A quick inspection shows that we have $D_1=\emptyset$, $D_2^<=\{1,2,2,3,3,3\}$, $D_2^>=\{4,4,4,5,5\}$, $D_2=\{1,2*2,3*3,3*4,2*5\}$, $C_1^< = \{2,2\}$, $C_1^>=\{3,4\}$, and so on.
\end{example}

The following, rather obvious lemma will be used to minimize the number of defects on the grid.

\begin{lemma}\label{Lm:Grid}
Let $n_1\le n_2$ be positive even integers. Consider $I(n_1,n_2)$ as a graph where two vertices $x$, $y$ are joined by an edge if an only if $\{x,y\}\in N$. Color a half of the vertices red and the other half blue. Then there are at least $n_1$ red-blue edges.

Suppose further that there are precisely $n_1$ red-blue edges. If $n_1<n_2$ then the grid can be evenly split by a vertical cut into blue and red connected components. If $n_1=n_2$ then the grid can be evenly split by a vertical cut or a horizontal cut into blue and red connected components.
\end{lemma}
\begin{proof}
Let us first show that the number $e$ of red-blue edges satisfies $e\ge n_1$. If no row is monochromatic then $e\ge n_1$. If no column is monochromatic then $e\ge n_2\ge n_1$. We can therefore assume that there exists a monochromatic row and a monochromatic column, without loss of generality forming a red cross. Let $r$ (resp. $c$) be the number of rows (resp. columns) containing a blue vertex. Any row or column containing a blue vertex intersects the red cross, hence $r+c\le e$. We also have $rc\ge n_1n_2/2$, the number of blue vertices. Therefore, $n_1^2\le n_1n_2\le 2rc<(r+c)^2\le e^2$, and $e>n_1$ follows.

Suppose that $e=n_1$. If no row is monochromatic then every column is monochromatic (else $e>n_1$) and we can see easily that the grid can be evenly split by a vertical cut into two monochromatic connected components. If there is a monochromatic row, then no column is monochromatic (else $e>n_1$ as in the first paragraph), which forces $n_1=n_2$, and we conclude the proof using a transposed argument.
\end{proof}

\begin{corollary}\label{Cr:Grid}
Every permutation of $I(2t_1,2t_2)$ contains at least $2t_1$ defects, i.e., $|D_1\cup D_2|\ge 2t_1$.
\end{corollary}
\begin{proof}
Color $(i,j)$ blue if $\pi(i,j)$ is homogeneous, and red otherwise. By Lemma \ref{Lm:Grid} there are at least $2t_1$ red-blue edges, each contributing a defect. For instance, an edge between $S_1^<\times S_2^<$ and $S_1^>\times S_2^<$ contributes a column defect.
\end{proof}

\subsection{An exact formula}

The value of $f(\pi)$ can be expressed exactly in terms of the corners, boundaries and defects. To that effect, let
\begin{equation}\label{Eq:Opt}
    g(\pi) = \sum_{i=1}^2\left(\left(\sum B_i^>+2\sum C_i^> + 2\sum D_i^>\right) -\left(\sum B_i^< + 2\sum C_i^< + 2\sum D_i^<\right)\right).
\end{equation}

\begin{proposition}\label{Pr:ExactWithH}
Let $\pi$ be a permutation of $I(2t_1,2t_2)$. Then
\begin{displaymath}
    f(\pi) = \sum_{i=1}^2\left(\sum S_i^> - \sum S_i^<\right) - g(\pi).
\end{displaymath}
\end{proposition}
\begin{proof}
Consider a cell $x$ with $\pi(x)=(i,j)$ and $j\in S_2^>$. The value $j$ is counted $8t_1$ times in $\sum S_2^>$. We must make the following modifications to $f(\pi)$ on account of $x$. If $x\in B$ then $j$ must be subtracted once from $\sum S_2^>$ because $x$ has only three neighbors. If $x\in C$ then $j$ must be subtracted twice from $\sum S_2^>$ because $x$ has only two neighbors. If $x$ participates in a defect in the second coordinate (as illustrated by the neighboring values $(1,4)$ and $(3,6)$ in Example \ref{Ex:Defect}), there is some neighbor $y$ of $x$ such that $\pi(y)=(i',j')$ with $j'\in S_2^>$. This defect contributes $k = \min\{j,j'\}$ to $D_2^>$, and we must subtract $k$ two times from $\sum S_2^>$ because of the sign change in the defect.

Similarly for $j\in S_2^<$, except that the value has to be added to $-\sum S_2^<$. Analogously for the first coordinate.
\end{proof}

Maximizing $f(\pi)$ is thus equivalent to minimizing the positive quantity $g(\pi)$. The following result shows that the number of positive summands in \eqref{Eq:Opt} is equal to the number of negative summands in \eqref{Eq:Opt}.

\begin{proposition}\label{Pr:Balance}
Let $\pi$ be a permutation of $I(2t_1,2t_2)$. Then
\begin{displaymath}
    |B_i^>| + 2|C_i^>| + 2|D_i^>| = |B_i^<| + 2|C_i^<| + 2|D_i^<|
\end{displaymath}
for $i\in\{1,2\}$.
\end{proposition}
\begin{proof}
Throughout the proof, if we deal with two permutations $\pi$ and $\pi'$, then $B_i^>$ and similar multisets will refer to $\pi$, while ${B'}_i^>$ and similar multisets will refer to $\pi'$.

We will prove the result for $i=1$. The case $i=2$ is similar. Place the grid on a torus by identifying the opposite sides, and interpret the result as having no corners and boundaries, i.e., $C_1=B_1=\emptyset$. Fix any initial permutation on the torus for which small and large values alternate in the first coordinate, that is, $D_1=\emptyset$. Such a permutation exists since both dimensions are even.

Any other permutation on the torus is obtained from the initial permutation by a sequence of successive swaps of two cells. Suppose that the result holds for $\pi$ and let $\pi'$ be the permutation obtained from $\pi$ by swapping yet another pair of cells $x$, $y$. If $\pi(x)_1$, $\pi(y)_1$ are both small or both large, then $|{D'}_1^<|=|D_1^<|$ and $|{D'}_1^>|=|D_1^>|$. Otherwise let us assume without loss of generality that $(\pi(x)_1,\pi(y)_1)\in S_1^<\times S_1^>$. Let $z_x^<$ (resp.\;$z_x^>$) denote the number of neighbors of $x$ in the grid of $\pi$ that have a small entry (resp.\;a large entry) in the first coordinate, and define $z_y^<$, $z_y^>$ similarly. Then $|{D'}_1^<| = |D_1^<| - z_x^< + z_y^<$ and $|{D'}_1^>| = |D_1^>| - z_y^> + z_x^>$. Thus $|{D'}_1^<| - |{D'}_1^>| = |D_1^<| - |D_1^>| - z_x^< - z_x^> + z_y^< + z_y^> = |D_1^<| - |D_1^>|$ because $z_x^< + z_x^> = z_y^< + z_y^> = 4$ on the torus. The result therefore holds for any permutation on the torus.

Any permutation of the grid is obtained from some permutation on the torus by cutting between certain two rows and between certain two columns. Let $\pi$ be a permutation on the torus, and let $\pi'$ be a permutation on a cylinder obtained from $\pi$ by cutting between two rows. Let $x$, $y$ be two cells that are neighbors in $\pi$ but not in $\pi'$. Let us denote by $b_1^<(x,y)$ the contribution of the cells $x$, $y$ to $|B_1^<|$, and similarly for other quantities. We have ${c'}_1^<(x,y) = c_1^<(x,y) = 0 = {c'}_1^>(x,y) = c_1^>(x,y)$ because there are still no corners in $\pi'$. Suppose that $(\pi(x)_1,\pi(y)_1)\in S_1^<\times S_1^<$. Then ${d'}_1^>(x,y)=d_1^>(x,y)$, ${b'}_1^>(x,y) = b_1^>(x,y)$, while ${d'}_1^<(x,y) = d_1^<(x,y)-1$ (because the small row defect between $x$ and $y$ has been removed), and ${b'}_1^<(x,y) = b_1^<(x,y)+2$ (because $x$, $y$ are on the boundary of $\pi'$). Similarly if $(\pi(x)_1,\pi(y)_1)\in S_1^>\times S_1^>$. Suppose without loss of generality that $(\pi(x)_1,\pi(y)_1)\in S_1^<\times S_1^>$. Then ${d'}_1^<(x,y) = d_1^<(x,y)$, ${d'}_1^>(x,y) = d_1^>(x,y)$, ${b'}_1^<(x,y) = b_1^<(x,y)+1$ and ${b'}_1^>(x,y) = {b'}_1^>(x,y)+1$. Similarly for all other pairs of cells across the cut, so the result holds for $\pi'$.

Finally, let $\pi$ be a permutation on the cylinder, and let $\pi'$ be obtained from $\pi$ by cutting between two columns. Let $x$, $y$ be two cells that are neighbors in $\pi$ but not in $\pi'$. If $x$, $y$ are not among the corners of $\pi'$ then we can argue as above. Thus suppose that $x$, $y$ are among the corners of $\pi'$, necessarily in the same row. Suppose that $(\pi(x)_1,\pi(y)_1)\in S_1^<\times S_1^<$. Then ${c'}_1^<(x,y) = c_1^<(x,y) + 2$, ${b'}_1^<(x,y) = b_1^<(x,y)-2$ (because both $x$, $y$ were on the boundary of $\pi$ but are among the corners of $\pi'$), ${d'}_1^<(x,y) = d_1^<(x,y)-1$, and the large values are not affected. Similarly if $(\pi(x)_1,\pi(y)_1)\in S_1^>\times S_1^>$. Suppose without loss of generality that $(\pi(x)_1,\pi(y)_1)\in S_1^<\times S_1^>$. Then ${c'}_1^<(x,y) = c_1^<(x,y)+1$, ${c'}_1^>(x,y) = c_1^>(x,y)+1$, ${b'}_1^<(x,y) = b_1^<(x,y)-1$, ${b'}_1^>(x,y) = b_1^>(x,y) - 1$, and the defect values are not affected. Hence the result holds for $\pi'$, and we are through.
\end{proof}

\subsection{Separating the contributions of corners, boundaries and defects}\label{Ss:Separating}

Recall that maximizing $f(\pi)$ is equivalent to minimizing $g(\pi)$. The difficulty at hand is that the row and column coordinates cannot be optimized independently because the multisets $B$, $C$ consist of values on cells, not of independent row and column values. Another difficulty is that the corners and defects are counted twice in $g(\pi)$. In this subsection we derive a crucial inequality for $g(\pi)$ that removes most of the difficulties.

For a permutation $\pi$ of $I(n_1,n_2)$, let
\begin{displaymath}
    h(\pi) = \sum_{i=1}^2\left(\sum B_i^>+ \sum C_i^>- \sum B_i^<- \sum C_i^<\right).
\end{displaymath}
Furthermore, for $i\in\{1,2\}$ let
\begin{displaymath}
    x_i = |B_i^<|+|C_i^<|-|B_i^>|-|C_i^>|.
\end{displaymath}

\begin{proposition}\label{Pr:Key}
Let $\pi$ be a permutation of $I(n_1,n_2)$ with $n_i=2t_i$ and $n_1\le n_2$. Then
\begin{displaymath}
    g(\pi)\ge h(\pi) + 4 + n_1 + \sum_{i=1}^2 (t_i+1/2)x_i.
\end{displaymath}
\end{proposition}
\begin{proof}
By Proposition \ref{Pr:Balance},
\begin{displaymath}
    |D_i^>| - |D_i^<| = (|B_i^<|+2|C_i^<|-|B_i^>|-2|C_i^>|)/2 = x_i/2 + (|C_i^<|-|C_i^>|)/2,
\end{displaymath}
and hence
\begin{align}
    \sum D_i^> - \sum D_i^< &\ge (t_i+1)|D_i^>| - t_i|D_i^<| \notag\\
        &= (t_i+1/2)(|D_i^>|-|D_i^<|) + (|D_i^>|+|D_i^<|)/2\notag\\
        &= (t_i+1/2)(|D_i^>|-|D_i^<|) + |D_i|/2\label{Eq:IneqD}\\
        &= (t_i+1/2)x_i/2 + (t_i+1/2)(|C_i^<|-|C_i^>|)/2 + |D_i|/2.\notag
\end{align}
Moreover,
\begin{align}
    (t_i+1/2)&(|C_i^<|-|C_i^>|) + \sum C_i^> - \sum C_i^<\notag \\
    &=\left(\sum C_i^> - (t_i+1/2)|C_i^>|\right) + \left((t_i+1/2)|C_i^<| - \sum C_i^<\right)\label{Eq:IneqC}\\
    &\ge |C_i^>|/2 + |C_i^<|/2 = |C_i|/2=2.\notag
\end{align}
Using \eqref{Eq:IneqD} and \eqref{Eq:IneqC}, we therefore have
\begin{align*}
    g(\pi) &= h(\pi) + \sum_{i=1}^2\left(2\sum D_i^>-2\sum D_i^< + \sum C_i^> - \sum C_i^<\right)\\
        &\ge h(\pi) + \sum_{i=1}^2\left((t_i+1/2)x_i + (t_i+1/2)(|C_i^<|-|C_i^>|) + \sum C_i^> - \sum C_i^< + |D_i|\right)\\
        &\ge h(\pi) + \sum_{i=1}^2 ((t_i+1/2)x_i + 2 + |D_i| )\\
        &= h(\pi) + 4 + |D_1\cup D_2| + \sum_{i=1}^2(t_i+1/2)x_i.
\end{align*}
Finally, by Corollary \ref{Cr:Grid} we have
\begin{equation}\label{Eq:Ineqd}
    |D_1\cup D_2|\ge n_1,
\end{equation}
finishing the proof.
\end{proof}

The inequality of Proposition \ref{Pr:Key} is an equality if and only if the three inequalities \eqref{Eq:IneqD}, \eqref{Eq:IneqC}, \eqref{Eq:Ineqd} are equalities. As we shall see in Section \ref{Sc:Optimal}, these equalities are simultaneously attainable if $n_1$, $n_2$ are large enough.

\subsection{Best approximations to disks}

Let $n_1$, $n_2$ be positive integers, and let
\begin{displaymath}
    O=((n_1+1)/2,(n_2+1)/2)
\end{displaymath}
be the geometric center of $I(n_1,n_2)$. For a subset $A$ of $I(n_1,n_2)$ let
\begin{displaymath}
    w(A) = \sum_{x\in A} d(x,O),
\end{displaymath}
and for a positive integer $k$ let
\begin{displaymath}
    w_k(n_1,n_2) = \min\setof{w(A)}{A\subseteq I(n_1,n_2),\,|A|=k}.
\end{displaymath}

A subset $A$ of $I(n_1,n_2)$ is said to be a \emph{(best approximation to a) disk} if $w(A)= w_{|A|}(n_1,n_2)$.

In order to understand disks in $I$, define the $i$th \emph{layer} $\Lambda_i$ by
\begin{displaymath}
    \Lambda_i=\setof{x\in I(n_1,n_2)}{d(x,O)=i}.
\end{displaymath}
Figure \ref{Fg:Layers} depicts the layers $\Lambda_1$, $\Lambda_2$, $\Lambda_3$ in the situation when $n_1=4$ and $n_2\ge 6$ is even.

\begin{figure}[h]
\begin{tikzpicture}[scale=0.6]
\draw[very thick] (-4,0)--(4,0);
\draw[very thick] (0,-2)--(0,2);
\draw (-1,-2)--(-1,2);
\draw (1,-2)--(1,2);
\draw (-2,-2)--(-2,2);
\draw (2,-2)--(2,2);
\draw (-3,-1)--(-3,1);
\draw (3,-1)--(3,1);
\draw (-3,1)--(3,1);
\draw (-3,-1)--(3,-1);
\draw (-2,2)--(2,2);
\draw (-2,-2)--(2,-2);
\node at (0.5,0.5) {$1$};
\node at (-0.5,0.5) {$1$};
\node at (0.5,-0.5) {$1$};
\node at (-0.5,-0.5) {$1$};
\node at (0.5,1.5) {$2$};
\node at (-0.5,1.5) {$2$};
\node at (0.5,-1.5) {$2$};
\node at (-0.5,-1.5) {$2$};
\node at (1.5,0.5) {$2$};
\node at (-1.5,0.5) {$2$};
\node at (1.5,-0.5) {$2$};
\node at (-1.5,-0.5) {$2$};
\node at (1.5,1.5) {$3$};
\node at (-1.5,1.5) {$3$};
\node at (1.5,-1.5) {$3$};
\node at (-1.5,-1.5) {$3$};
\node at (2.5,0.5) {$3$};
\node at (-2.5,0.5) {$3$};
\node at (2.5,-0.5) {$3$};
\node at (-2.5,-0.5) {$3$};
\end{tikzpicture}
\caption{The first three layers around the geometric center of $I(4,n_2)$.}
\label{Fg:Layers}
\end{figure}

\begin{lemma}\label{Lm:Disk}
A subset $A$ of the grid $I(n_1,n_2)$ is a disk if and only if $A=P\cup \bigcup_{i=1}^r \Lambda_i$ for some $r\ge 0$ and $P\subseteq \Lambda_{r+1}$.
\end{lemma}
\begin{proof}
Suppose that there are $i<j$, $x\in \Lambda_i$, $y\in \Lambda_j$ such that $x\not\in A$ and $y\in A$. Then $A' = (A\cup \{x\})\setminus \{y\}$ satisfies $|A'|=|A|$ and $w(A') =  w(A) + i - j < w(A)$, a contradiction. The result follows.
\end{proof}

Before we return to the problem of neighbor separation, let us explicitly find the values $w_k(n_1,n_2)$ when $n_1=2t_1$, $n_2=2t_2$ are even and $n_1\le n_2$, using Lemma \ref{Lm:Disk}. The formula is straightforward but a bit cumbersome. First, we have
\begin{displaymath}
    |\Lambda_i| = \left\{\begin{array}{ll}
        4i,&\text{ if }0\le i\le t_1,\\
        2n_1,&\text{ if }t_1<i\le t_2,\\
        2n_1 - 4(i-t_2) = 2(n_1+n_2-2i),&\text{ if }t_2<i\le t_1+t_2,\\
        0,&\text{ if }t_1+t_2<i.
    \end{array}\right.
\end{displaymath}
Given $k$ such that $0\le k\le n_1n_2$, let $r$ be the largest integer such that $\sum_{i=1}^r |\Lambda_i|\le k$. Then, by Lemma \ref{Lm:Disk},
\begin{equation}\label{Eq:WeightedBall}
    w_k(n_1,n_2) = \sum_{i=1}^r i|\Lambda_i| + (r+1)\left(k-\sum_{i=1}^r|\Lambda_i|\right) = (r+1)k - \sum_{i=1}^r(r+1-i)|\Lambda_i|.
\end{equation}

\medskip

Here is the key connection between permutations and disks in $I(n_1,n_2)$, with $h(\pi)$ and $x_i$ as in Subsection \ref{Ss:Separating}.

\begin{lemma}\label{Lm:Ball}
Let $\pi$ be a permutation of $I(n_1,n_2)$ with $n_i=2t_i$. Then
\begin{displaymath}
    w(B\cup C) = h(\pi) + \sum_{i=1}^2(t_i+1/2)x_i.
\end{displaymath}
\end{lemma}
\begin{proof}
Consider a pair of cells in $B\cup C$, one with a small row value, say $(t_1-i,u)$, and one with a large row value, say $(t_1+j,v)$. Then the contribution of the first coordinates of these two cells to $w(B\cup C)$ is $((t_1+1/2)-(t_1-i)) + ((t_1+j)-(t_1+1/2)) = i+j$. On the other hand, the analogous contribution to $h(\pi)$ is $t_1+j - (t_1-i) = i+j$. Similarly for small and large column values.

Now consider all cells of $B\cup C$. Pair together as many first coordinates of these cells as possible, subject to the constraint that small row values are matched with large row values. Pair column values similarly. For these matched pairs of coordinates, the contribution to $w(B\cup C)$ and $h(\pi)$ is the same, as we have just shown. If $x_1\ge 0$, there are $x_1$ small row values remaining (unmatched), otherwise there are $-x_1$ large row values remaining. Similarly for columns.

Let the $|x_1|=q$ unmatched row values be $i_1$, $\dots$, $i_q$. If $x_1\ge 0$, the contribution of these values to $h(\pi)$ is $-\sum_{j=1}^q i_j$, while the contribution to $w(B\cup C)$ is $\sum_{j=1}^q ((t_1+1/2)-i_j) = -\sum_{j=1}^q i_j + (t_1+1/2)q = - \sum_{j=1}^q i_j + (t_1+1/2)x_1$. If $x_1\le 0$, the contribution of these values to $h(\pi)$ is $\sum_{j=1}^q i_j$, while the contribution to $w(B\cup C)$ is $\sum_{j=1}^q (i_j - (t_1+1/2)) = \sum_{j=1}^q i_j + (t_1+1/2)x_1$. Similarly for columns, and the result follows.
\end{proof}

Since $|B\cup C| = 2n_1+2n_2-4$, we deduce from Proposition \ref{Pr:Key} and Lemma \ref{Lm:Ball}:

\begin{corollary}\label{Cr:Key}
Let $\pi$ be a permutation of $I(n_1,n_2)$ with $n_i=2t_i$ and $n_1\le n_2$. Then
\begin{displaymath}
    g(\pi)\ge 4+n_1+w(B\cup C).
\end{displaymath}
In particular,
\begin{equation}\label{Eq:Key}
    g(\pi)\ge 4+n_1+w_{2n_1+2n_2-4}(n_1,n_2).
\end{equation}
\end{corollary}

\section{Optimal solutions}\label{Sc:Optimal}

In this section we construct a family of optimal solutions and prove the main result, Theorem \ref{Th:Main}.

\subsection{A family of optimal solutions}

Suppose that $2t_1=n_1\le n_2=2t_2$. Note that the inequality \eqref{Eq:Key} for $\pi$ is in fact an equality if and only if all of the following conditions hold:
\begin{enumerate}
\item[$\bullet$] \eqref{Eq:IneqC} is an equality, that is, every value of $C_i^>$ is equal to $t_i+1$ and every value of $C_i^<$ is equal to $t_i$, which is equivalent to $C=\Lambda_1$.
\item[$\bullet$] \eqref{Eq:Ineqd} is an equality, that is, $|D_1\cup D_2|=n_1$. By Lemma \ref{Lm:Grid}, this is equivalent to the grid of $\pi$ being evenly split into a homogeneous half and a heterogeneous half, with all defects occurring across the dividing line. In particular, $x_1=x_2=0$. The actual placement of entries of $\Lambda_1$ into the corners determines the direction of the cut (we only have a choice if $n_1=n_2$) and whether all defects will be row defects or column defects. (See Example \ref{Ex:Sigma}.)
\item[$\bullet$] \eqref{Eq:IneqD} is an equality, that is, every value of $D_i^>$ is equal to $t_i+1$ and every value of $D_i^<$ is equal to $t_i$. We observe that the defects must then be either all row defects alternating between $t_1$ and $t_1+1$, or all column defects alternating between $t_2$ and $t_2+1$, as the following diagram illustrates:
    \begin{displaymath}
        \begin{array}{c|c}
            (<,<)&(<,>)\\
            (>,>)&(>,<)\\
            \vdots&\vdots
        \end{array}
        \quad\quad{\text{or}}\quad\quad
        \begin{array}{c|c}
            (<,<)&(>,<)\\
            (>,>)&(<,>)\\
            \vdots&\vdots
        \end{array}
    \end{displaymath}
\item[$\bullet$] $w(B\cup C) = w_{2n_1+2n_2-4}(n_1,n_2)$, that is, $B\cup C$ is a disk in $I(n_1,n_2)$.
\end{enumerate}

Construction \ref{Co:Pi} below yields the family $\mathcal P(n_1,n_2)$ of all permutations on $I(n_1,n_2)$ that satisfy all these conditions. For certain values of $n_1$ and $n_2$, steps (ii) and (iii) of the construction might not be completable due to the lack of suitable defect entries, and in some additional cases the construction can be completed with row defects but not with column defects (since $n_1\le n_2$). Nevertheless, it will be clear from the construction that $\mathcal P(n_1,n_2)$ is nonempty when $n_1$, $n_2$ are sufficiently large. We will give exact bounds on $n_1$, $n_2$ in Subsection \ref{Ss:Bounds}.

\begin{construction}\label{Co:Pi}
Let $2t_1=n_1\le n_2=2t_2$. The following procedure constructs all permutations $\pi$ of $\mathcal P(n_1,n_2)$ with a vertical defect line. (When $n_1=n_2$, the members of $\mathcal P(n_1,n_2)$ with a horizontal defect line are obtained by a transposed construction. When $n_1<n_2$, no members of $\mathcal P(n_1,n_2)$ have a horizontal defect line.)
\begin{enumerate}
\item[(i)] Populate the four corners of $\pi$ with cells from $\Lambda_1$ so that one half (left or right) of the grid is homogeneous and the other half is heterogeneous. This will determine if all defects will be row defects or column defects.
\item[(ii)] Populate the boundary with entries from $\Lambda_2$, $\Lambda_3$, and so on, starting with $\Lambda_2$, always filling in four entries at a time, using one entry from each $S_1^>\times S_2^>$, $S_1^>\times S_2^<$, $S_1^<\times S_2^>$, $S_1^<\times S_2^<$. (Note that the double checkerboard forces $B\cup C$ to have the same number of entries of each of the $4$ kinds.) In case of row defects, make sure that the two defects in the top and bottom row across the defect line have values $t_1$ and $t_1+1$. In case of column defects, the defect values must be $t_2$ and $t_2+1$.
\item[(iii)] Populate half of the interior cells along the defect line, one in each row, so that the row defects alternate values $t_1$, $t_1+1$ (or the column defects alternate values $t_2$, $t_2+1$).
\item[(iv)] Populate the remaining cells arbitrarily, subject only to the restrictions imposed by the double checkerboard.
\end{enumerate}
\end{construction}

Let us illustrate Construction \ref{Co:Pi} with three examples that will be helpful in Subsection \ref{Ss:Bounds}. In all examples the partial grids are labeled according to the steps in Construction \ref{Co:Pi}. To save space and to improve legibility, we write $i,j$ instead of $(i,j)$ as the values of $\pi$.

\def\co{$+$}            
\def\bo{$-$}            
\def\de{$\square$}      
\def\bd{$\boxminus$}    
\def\cw{0.5cm} 

\begin{example}\label{Ex:Sigma}
Let $n_1=n_2=6$. The four grids in Figure \ref{Fg:OptimalSigma} show how to construct a permutation in $\mathcal P(n_1,n_2)$ with a vertical defect line and row defects.

In step (i) we have made the choice to have a vertical defect line (by placing the homogeneous values in the same column rather than in the same row), to have homogeneous values on the left, and to have row defects rather than column defects (by not transposing the locations of the entries $(3,4)$ and $(4,3)$, say). It so happens that any of the $16 = 4\cdot 2\cdot 2$ valid corner placements are completable here.

\begin{figure}[ht]
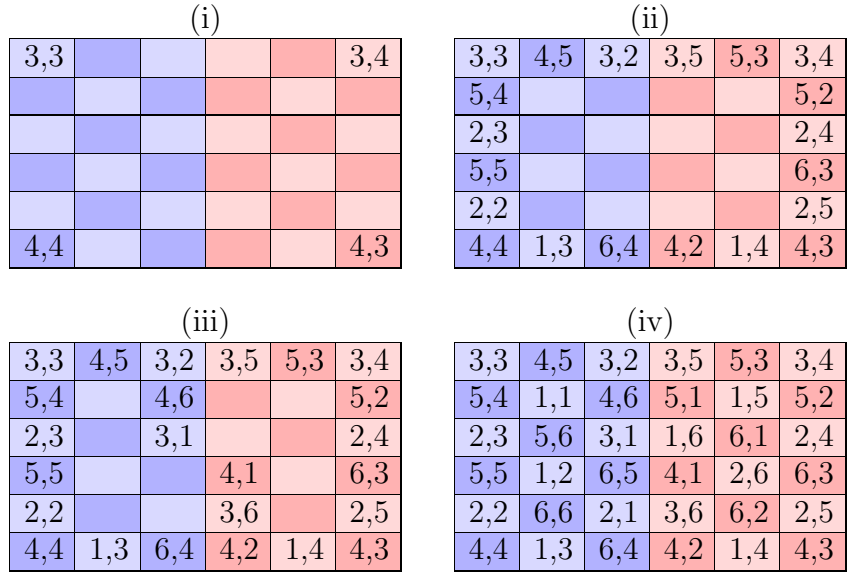

\begin{displaymath}
    \begin{array}{cc}
        \begin{array}{c}
            \text{(i)}\\
            \begin{tabular}{|p{\cw}|p{\cw}|p{\cw}|p{\cw}|p{\cw}|p{\cw}|}
                \hline \cll3,3&\chh&\cll&\clh&\chl&\clh3,4\\
                \hline \chh&\cll&\chh&\chl&\clh&\chl\\
                \hline \cll&\chh&\cll&\clh&\chl&\clh\\
                \hline \chh&\cll&\chh&\chl&\clh&\chl\\
                \hline \cll&\chh&\cll&\clh&\chl&\clh\\
                \hline \chh4,4&\cll&\chh&\chl&\clh&\chl4,3\\
                \hline
            \end{tabular}
        \end{array}
        &
        \begin{array}{c}
            \text{(ii)}\\
            \begin{tabular}{|p{\cw}|p{\cw}|p{\cw}|p{\cw}|p{\cw}|p{\cw}|}
                \hline \cll3,3&\chh4,5&\cll3,2&\clh3,5&\chl5,3&\clh3,4\\
                \hline \chh5,4&\cll&\chh&\chl&\clh&\chl5,2\\
                \hline \cll2,3&\chh&\cll&\clh&\chl&\clh2,4\\
                \hline \chh5,5&\cll&\chh&\chl&\clh&\chl6,3\\
                \hline \cll2,2&\chh&\cll&\clh&\chl&\clh2,5\\
                \hline \chh4,4&\cll1,3&\chh6,4&\chl4,2&\clh1,4&\chl4,3\\
                \hline
            \end{tabular}
        \end{array}
        \\
        \\
        \begin{array}{c}
            \text{(iii)}\\
            \begin{tabular}{|p{\cw}|p{\cw}|p{\cw}|p{\cw}|p{\cw}|p{\cw}|}
                \hline \cll3,3&\chh4,5&\cll3,2&\clh3,5&\chl5,3&\clh3,4\\
                \hline \chh5,4&\cll&\chh4,6&\chl&\clh&\chl5,2\\
                \hline \cll2,3&\chh&\cll3,1&\clh&\chl&\clh2,4\\
                \hline \chh5,5&\cll&\chh&\chl4,1&\clh&\chl6,3\\
                \hline \cll2,2&\chh&\cll&\clh3,6&\chl&\clh2,5\\
                \hline \chh4,4&\cll1,3&\chh6,4&\chl4,2&\clh1,4&\chl4,3\\
                \hline
            \end{tabular}
        \end{array}
        &
        \begin{array}{c}
            \text{(iv)}\\
            \begin{tabular}{|p{\cw}|p{\cw}|p{\cw}|p{\cw}|p{\cw}|p{\cw}|}
                \hline \cll3,3&\chh4,5&\cll3,2&\clh3,5&\chl5,3&\clh3,4\\
                \hline \chh5,4&\cll1,1&\chh4,6&\chl5,1&\clh1,5&\chl5,2\\
                \hline \cll2,3&\chh5,6&\cll3,1&\clh1,6&\chl6,1&\clh2,4\\
                \hline \chh5,5&\cll1,2&\chh6,5&\chl4,1&\clh2,6&\chl6,3\\
                \hline \cll2,2&\chh6,6&\cll2,1&\clh3,6&\chl6,2&\clh2,5\\
                \hline \chh4,4&\cll1,3&\chh6,4&\chl4,2&\clh1,4&\chl4,3\\
                \hline
            \end{tabular}
        \end{array}
    \end{array}
\end{displaymath}
\caption{A permutation of the $6\times 6$ grid maximizing the separation of neighbors with respect to the $L_1$ metric.}\label{Fg:OptimalSigma}
\end{figure}

The following diagram corresponds to our choices of corner cells (\co), boundary cells (\bo), cells that contribute a defect (\de), and cells that are both on the boundary and contribute a defect (\bd). Note that we have $C=\Lambda_1$ and $B\cup C$ is a disk. Also note that in the row defect resulting from the neighboring cells $(3,2)$, $(3,5)$ we have arbitrarily assigned $(3,2)$ as the cell contributing the defect.
\begin{displaymath}
\begin{tikzpicture}[scale=0.5]
    \draw[thick] (-3,0)--(3,0);
    \draw[thick] (0,-3)--(0,3);
    \node at (-2.5,3.5) {$1$}; \node at (-1.5,3.5) {$2$}; \node at (-0.5,3.5) {$3$}; \node at (0.5,3.5) {$4$}; \node at (1.5,3.5) {$5$}; \node at (2.5,3.5) {$6$};          
    \node at (-3.5,2.5) {$1$}; \node at (-3.5,1.5) {$2$}; \node at (-3.5,0.5) {$3$}; \node at (-3.5,-0.5) {$4$}; \node at (-3.5,-1.5) {$5$}; \node at (-3.5,-2.5) {$6$};    
    \node at (0.5,0.5) {\co}; \node at (-0.5,0.5) {\co}; \node at (0.5,-0.5) {\co}; \node at (-0.5,-0.5) {\co}; 
    \node at (1.5,0.5) {\bo}; \node at (-1.5,0.5) {\bd}; \node at (1.5,-0.5) {\bo}; \node at (-1.5,-0.5) {\bd}; 
    \node at (0.5,1.5) {\bo}; \node at (-0.5,1.5) {\bo}; \node at (0.5,-1.5) {\bo}; \node at (-0.5,-1.5) {\bo};
    \node at (2.5,0.5) {\de}; \node at (-2.5,0.5) {\de}; \node at (2.5,-0.5) {\de}; \node at (-2.5,-0.5) {\de}; 
    \node at (1.5,1.5) {\bo}; \node at (-1.5,1.5) {\bo}; \node at (1.5,-1.5) {\bo}; \node at (-1.5,-1.5) {\bo};
    \node at (0.5,2.5) {\bo}; \node at (-0.5,2.5) {\bo}; \node at (0.5,-2.5) {\bo}; \node at (-0.5,-2.5) {\bo};
\end{tikzpicture}
\end{displaymath}
\end{example}

\begin{example}\label{Ex:Sigma2}
Let $n_1=4$ and $n_2=12$. Using the same notation as in Example \ref{Ex:Sigma}, here is a diagram for a permutation in $\mathcal P(4,12)$
\begin{displaymath}
\begin{tikzpicture}[scale=0.5]
    \draw[thick] (-6,0)--(6,0);
    \draw[thick] (0,-2)--(0,2);
    \node at (-5.5,2.5) {$1$}; \node at (-4.5,2.5) {$2$}; \node at (-3.5,2.5) {$3$}; \node at (-2.5,2.5) {$4$}; \node at (-1.5,2.5) {$5$}; \node at (-0.5,2.5) {$6$}; 
    \node at (0.5,2.5) {$7$}; \node at (1.5,2.5) {$8$}; \node at (2.5,2.5) {$9$}; \node at (3.5,2.5) {$10$}; \node at (4.5,2.5) {$11$}; \node at (5.5,2.5) {$12$};
    \node at (-6.5,1.5) {$1$}; \node at (-6.5,0.5) {$2$}; \node at (-6.5,-0.5) {$3$}; \node at (-6.5,-1.5) {$4$}; 
    \node at (0.5,0.5) {\co}; \node at (-0.5,0.5) {\co}; \node at (0.5,-0.5) {\co}; \node at (-0.5,-0.5) {\co}; 
    \node at (1.5,0.5) {\bo}; \node at (-1.5,0.5) {\bo}; \node at (1.5,-0.5) {\bo}; \node at (-1.5,-0.5) {\bo}; 
    \node at (0.5,1.5) {\bo}; \node at (-0.5,1.5) {\bo}; \node at (0.5,-1.5) {\bo}; \node at (-0.5,-1.5) {\bo};
    \node at (2.5,0.5) {\bo}; \node at (-2.5,0.5) {\bo}; \node at (2.5,-0.5) {\bo}; \node at (-2.5,-0.5) {\bo}; 
    \node at (1.5,1.5) {\bo}; \node at (-1.5,1.5) {\bo}; \node at (1.5,-1.5) {\bo}; \node at (-1.5,-1.5) {\bo};
    \node at (3.5,0.5) {\bo}; \node at (-3.5,0.5) {\bd}; \node at (3.5,-0.5) {\bo}; \node at (-3.5,-0.5) {\bd}; 
    \node at (2.5,1.5) {\bo}; \node at (-2.5,1.5) {\bo}; \node at (2.5,-1.5) {\bo}; \node at (-2.5,-1.5) {\bo};
    \node at (-4.5,0.5) {\de}; \node at (-4.5,-0.5) {\de}; 
\end{tikzpicture}
\end{displaymath}
Figure \ref{Fg:OptimalSigma2} gives a partial grid of a permutation in $\mathcal P(4,12)$ obtained by steps (i)--(iii) of Construction \ref{Co:Pi}. Step (iv) can be routinely completed. Note that there is no solution with column defects because $B\cup C = \bigcup_{i=1}^4\Lambda_i$ and there are no suitable column defect entries left for the interior rows.
\end{example}

\renewcommand\cw{0.75cm} 
\begin{figure}[ht]
\begin{displaymath}
    \begin{array}{c}
        \text{(i)--(iii)}\\
    \begin{tabular}{|p{\cw}|p{\cw}|p{\cw}|p{\cw}|p{\cw}|p{\cw}|p{\cw}|p{\cw}|p{\cw}|p{\cw}|p{\cw}|p{\cw}|}
        \hline \clh2,7&\chl4,6&\clh2,8&\chl3,5&\clh1,7&\chl3,3&\chh3,9&\cll2,5&\chh4,7&\cll1,6&\chh3,8&\cll2,6\\
        \hline \chl3,4&\clh&\chl&\clh&\chl&\clh&\cll2,2&\chh&\cll&\chh&\cll&\chh4,8\\
        \hline \clh1,8&\chl&\clh&\chl&\clh&\chl3,2&\chh&\cll&\chh&\cll&\chh&\cll2,4\\
        \hline \chl3,6&\clh2,9&\chl4,5&\clh1,9&\chl4,4&\clh2,10&\cll2,3&\chh4,9&\cll1,4&\chh3,10&\cll1,5&\chh3,7\\
        \hline
    \end{tabular}
    \end{array}
\end{displaymath}
\caption{A partial, routinely completable permutation of the $4\times 12$ grid maximizing the separation of neighbors with respect to the $L_1$ metric.}\label{Fg:OptimalSigma2}
\end{figure}

\begin{example}\label{Ex:Sigma3}
Let $n_1=10$ and $n_2=12$. Using the same notation as in Example \ref{Ex:Sigma}, here is a diagram for a permutation in $\mathcal P(10,12)$
\begin{displaymath}
\begin{tikzpicture}[scale=0.5]
    \draw[thick] (-6,0)--(6,0);
    \draw[thick] (0,-5)--(0,5);
    \node at (-5.5,5.5) {$1$}; \node at (-4.5,5.5) {$2$}; \node at (-3.5,5.5) {$3$}; \node at (-2.5,5.5) {$4$}; \node at (-1.5,5.5) {$5$}; \node at (-0.5,5.5) {$6$}; 
    \node at (0.5,5.5) {$7$}; \node at (1.5,5.5) {$8$}; \node at (2.5,5.5) {$9$}; \node at (3.5,5.5) {$10$}; \node at (4.5,5.5) {$11$}; \node at (5.5,5.5) {$12$};
    \node at (-6.5,4.5) {$1$}; \node at (-6.5,3.5) {$2$}; \node at (-6.5,2.5) {$3$}; \node at (-6.5,1.5) {$4$}; \node at (-6.5,0.5) {$5$};  
    \node at (-6.5,-0.5) {$6$}; \node at (-6.5,-1.5) {$7$}; \node at (-6.5,-2.5) {$8$}; \node at (-6.5,-3.5) {$9$}; \node at (-6.5,-4.5) {$10$};
    \node at (0.5,0.5) {\co}; \node at (-0.5,0.5) {\co}; \node at (0.5,-0.5) {\co}; \node at (-0.5,-0.5) {\co}; 
    \node at (1.5,0.5) {\bo}; \node at (-1.5,0.5) {\bo}; \node at (1.5,-0.5) {\bo}; \node at (-1.5,-0.5) {\bo}; 
    \node at (0.5,1.5) {\bo}; \node at (-0.5,1.5) {\bo}; \node at (0.5,-1.5) {\bo}; \node at (-0.5,-1.5) {\bo};
    \node at (2.5,0.5) {\bo}; \node at (-2.5,0.5) {\bo}; \node at (2.5,-0.5) {\bo}; \node at (-2.5,-0.5) {\bo}; 
    \node at (1.5,1.5) {\bo}; \node at (-1.5,1.5) {\bo}; \node at (1.5,-1.5) {\bo}; \node at (-1.5,-1.5) {\bo};
    \node at (0.5,2.5) {\bo}; \node at (-0.5,2.5) {\bo}; \node at (0.5,-2.5) {\bo}; \node at (-0.5,-2.5) {\bo};
    \node at (3.5,0.5) {\bo}; \node at (-3.5,0.5) {\bd}; \node at (3.5,-0.5) {\bo}; \node at (-3.5,-0.5) {\bd}; 
    \node at (2.5,1.5) {\bo}; \node at (-2.5,1.5) {\bo}; \node at (2.5,-1.5) {\bo}; \node at (-2.5,-1.5) {\bo};
    \node at (1.5,2.5) {\bo}; \node at (-1.5,2.5) {\bo}; \node at (1.5,-2.5) {\bo}; \node at (-1.5,-2.5) {\bo};
    \node at (0.5,3.5) {\bo}; \node at (-0.5,3.5) {\bo}; \node at (0.5,-3.5) {\bo}; \node at (-0.5,-3.5) {\bo};
    \node at (4.5,0.5) {\de}; \node at (-4.5,0.5) {\de}; \node at (4.5,-0.5) {\de}; \node at (-4.5,-0.5) {\de}; 
    \node at (5.5,0.5) {\de}; \node at (-5.5,0.5) {\de}; \node at (5.5,-0.5) {\de}; \node at (-5.5,-0.5) {\de};
\end{tikzpicture}
\end{displaymath}
Figure \ref{Fg:OptimalSigma3} gives a partial grid of a permutation $\mathcal P(10,12)$ obtained by steps (i)--(iii) of Construction \ref{Co:Pi}. Step (iv) can be routinely completed. There is again no solution here with column defects since after populating $B\cup C$ we have only four suitable column defect entries left for the $8$ interior rows.
\end{example}

\renewcommand\cw{0.75cm} 
\begin{figure}[ht]
\begin{displaymath}
    \begin{array}{c}
        \text{(i)--(iii)}\\
    \begin{tabular}{|p{\cw}|p{\cw}|p{\cw}|p{\cw}|p{\cw}|p{\cw}|p{\cw}|p{\cw}|p{\cw}|p{\cw}|p{\cw}|p{\cw}|}
        \hline \cll5,6&\chh6,8&\cll5,5&\chh7,7&\cll4,6&\chh6,9&\chl6,3&\clh5,8&\chl6,5&\clh4,7&\chl7,6&\clh5,7\\
        \hline \chh7,8&\cll&\chh&\cll&\chh&\cll&\clh5,11&\chl&\clh&\chl&\clh&\chl6,4\\
        \hline \cll5,4&\chh&\cll&\chh&\cll&\chh6,11&\chl&\clh&\chl&\clh&\chl&\clh5,9\\
        \hline \chh8,7&\cll&\chh&\cll&\chh&\cll5,1&\clh&\chl&\clh&\chl&\clh&\chl7,5\\
        \hline \cll4,5&\chh&\cll&\chh&\cll&\chh&\chl6,1&\clh&\chl&\clh&\chl&\clh4,8\\
        \hline \chh6,10&\cll&\chh&\cll&\chh&\cll5,2&\clh&\chl&\clh&\chl&\clh&\chl8,6\\
        \hline \cll3,6&\chh&\cll&\chh&\cll&\chh&\chl6,2&\clh&\chl&\clh&\chl&\clh3,7\\
        \hline \chh7,9&\cll&\chh&\cll&\chh&\cll&\clh5,12&\chl&\clh&\chl&\clh&\chl7,4\\
        \hline \cll4,4&\chh&\cll&\chh&\cll&\chh6,12&\chl&\clh&\chl&\clh&\chl&\clh5,10\\
        \hline \chh6,7&\cll3,5&\chh8,8&\cll2,6&\chh9,7&\cll5,3&\clh2,7&\chl9,6&\clh3,8&\chl8,5&\clh4,9&\chl6,6\\
        \hline
    \end{tabular}
    \end{array}
\end{displaymath}
\caption{A partial, routinely completable permutation of the $10\times 12$ grid maximizing the separation of neighbors with respect to the $L_1$ metric.}\label{Fg:OptimalSigma3}
\end{figure}

\subsection{Existence of $\mathcal P(n_1,n_2)$}\label{Ss:Bounds}

Let $n_1\le n_2$ be positive even integers. Recall that $\mathcal P(n_1,n_2)$ consists of all permutations $\pi$ of $I(n_1,n_2)$ for which the inequality \eqref{Eq:Key} is in fact an equality. The following lemma characterizes those parameters $n_1$, $n_2$ for which $\mathcal P(n_1,n_2)\ne\emptyset$.

\begin{lemma}\label{Lm:Bound}
Let $2t_1=n_1\le n_2=2t_2$, and let $r$ be the smallest integer for which $\Lambda=\bigcup_{i=1}^r \Lambda_i$ satisfies $|\Lambda|\ge |B\cup C|$. Then $\mathcal P(n_1,n_2)\ne\emptyset$ if and only if one of the following conditions holds:
\begin{enumerate}
\item[(i)] $|\Lambda|=|B\cup C|$ and $r\le t_2-t_1/2+1/2$,
\item[(ii)] $|\Lambda|>|B\cup C|$ and $r\le t_2-t_1/2+3/2$.
\end{enumerate}
\end{lemma}
\begin{proof}
We will again only discuss permutations with a vertical defect line. In order to have $\mathcal P(n_1,n_2)\ne\emptyset$, $B\cup C$ must form a disk in $I(n_1,n_2)$. Note that our choice of $r$ is the smallest possible so that $B\cup C$ can be chosen as a subset of layers $\Lambda_1$, $\dots$, $\Lambda_r$. Defects must be chosen either in rows $t_1$, $t_1+1$, or in columns $t_2$, $t_2+1$. Row defects can be selected whenever column defects can be selected (since $n_1\le n_2$). We will therefore focus on row defects. Then $\mathcal P(n_1,n_2)\ne\emptyset$ if and only if it is possible to select $t_1$ defects in row $t_1$ (one of which is in $B\cup C$), and $t_1$ defects in row $t_1+1$ (one of which is again in $B\cup C$).

Suppose that $n_1=2$. Then no defects outside of $B\cup C$ are needed. On the other hand, $\Lambda=B\cup C$ is the entire grid, $r=t_2$, $t_1=1$, so the inequality $r\le t_2-t_1/2+1/2$ is satisfied. We can therefore assume that $n_1>2$.

If $|\Lambda|=|B\cup C|$ (as in Examples \ref{Ex:Sigma2} and \ref{Ex:Sigma3}), we must pick $n_1-2$ defects outside of $\Lambda$, so we demand $4(t_2-r)\ge n_1-2$, which is equivalent to $r\le t_2-t_1/2+1/2$.

Suppose that $|\Lambda|>|B\cup C|$ (as in Example \ref{Ex:Sigma}). Then there are precisely four additional candidates for defects in $\Lambda$. Suppose that $n_1=4$. We can then always select the four needed defects in $\Lambda$ and, on the other hand, $r\le t_2-t_1/2+3/2$ holds because $r\le t_2$ and $t_1=2$. We can therefore assume that $n_1>4$, pick six defects in $\Lambda$, and we need to have $4(t_2-r)\ge n_1-6$, which is equivalent to $r\le t_2-t_1/2+3/2$.
\end{proof}

\begin{proposition} \label{Pr:SpecialCasesL1}
Let $n_1$, $n_2$ be positive even integers. Then $\mathcal P(n_1,n_2)\ne\emptyset$ except in the cases $n_1=n_2\in\{4$, $8$, $12$, $16\}$.
\end{proposition}
\begin{proof}
Let $n_i=2t_i$, $t_2=t_1+u$, and let $r$ be the smallest integer such that $\Lambda=\bigcup_{i=1}^r \Lambda_i$ satisfies $|\Lambda|\ge |B\cup C|=4(t_1+t_2-1)= 4(2t_1+u-1)$. Recall that $|\Lambda_i|=4i$ if $1\le i\le t_1$ and $|\Lambda_i|=4t_1$ if $t_1<i\le t_2$.

If $t_1=1$ then $r=t_2$ and the conditions of Lemma \ref{Lm:Bound} are satisfied. We can therefore assume that $t_1>1$.

We claim that $r\le t_2$. Indeed, $\sum_{i=1}^{t_2}|\Lambda_i|\ge |B\cup C|$ if and only if $\sum_{i=1}^{t_1}|\Lambda_i| + \sum_{i=t_1+1}^{t_2}|\Lambda_i| = 2t_1(t_1+1) + 4t_1(t_2-t_1) \ge 4(t_1+t_2-1)$, i.e., $2t_2(t_1-1)\ge (t_1+2)(t_1-1)$, $2t_2\ge t_1+2$, which holds because $t_2\ge t_1>1$.

Case 1: Suppose that $r>t_1$ and $t_2<3t_1/2-1/2$. Then $2t_1(t_1+1)\le \sum_{i=1}^{r-1}|\Lambda_i|<|B\cup C|=4(t_1+t_2-1) < 4(t_1+3t_1/2-1/2-1) = 10t_1-6$, so $t_1^2-4t_1+3<0$ and $t_1=2$. From $t_2<3t_1/2-1/2$ we then deduce $t_2=2$, so $r=2$, a contradiction.

Case 2: Suppose that $r>t_1$ and $t_2\ge 3t_1/2-1/2$, the latter yielding $m=t_2-t_1/2+1/2\ge t_1$. If $r\le m$ then the conditions of Lemma \ref{Lm:Bound} are satisfied, so suppose that $r>m$. Then $2t_1(t_1+1) + 4t_1(m-t_1) = \sum_{i=1}^m|\Lambda_i| \le \sum_{i=1}^{r-1}|\Lambda_i| < 4(t_1+t_2-1)$, which is equivalent to $t_2(t_1-1)<(t_1-1)(t_1+1)$ and thus to $t_2<t_1+1$. Hence $t_2=t_1$, a contradiction with $r\le t_2$.

Case 3: Suppose that $r\le t_1$. We have $|\Lambda|=\sum_{i=1}^r 4i = 2r(r+1)$. The positive root of the quadratic equation $2r(r+1)=4(2t_1+u-1)$ is $\rho = (\sqrt{16t_1+8u-7}-1)/2$. If $|\Lambda|=|B\cup C|$, we have $\rho=r$, and we need to satisfy the inequality $\rho\le t_2-t_1/2+1/2 = u+t_1/2+1/2$ of Lemma \ref{Lm:Bound}(i). If $|\Lambda|<|B\cup C|$, we have $\rho < r\le \rho+1$, and the inequality $\rho+1\le t_2-t_1/2+3/2$ (which we note is the same as in the case $|\Lambda|=|B\cup C|$) is at least as strong as the inequality we need to satisfy in Lemma \ref{Lm:Bound}(ii). Now, the inequality $\rho\le u+t_1/2+1/2$ is equivalent to
\begin{displaymath}
    t_1^2+(4u-12)t_1+(4u^2+11)\ge 0,
\end{displaymath}
and it fails only in the cases ($u=1$ and $t_1=4$) and ($u=0$ and $2\le t_1\le 10$). Out of these, only the cases $t_1=t_2\in\{2,4,6,8\}$ actually fail the conditions of Lemma \ref{Lm:Bound}.
\end{proof}

\subsection{Main result}

We can now state and prove the main result. See the comments following Theorem \ref{Th:Main} for the exceptional cases $n_1=n_2\in\{4,8,12,16\}$.

\begin{theorem}\label{Th:Main}
Let $n_1\le n_2$ be positive even integers such that either $n_1<n_2$ or $n_1=n_2\not\in \{4$, $8$, $12$, $16\}$. Let $I(n_1,n_2) = \{1,\dots,n_1\}\times\{1,\dots,n_2\}$ be the $n_1\times n_2$ grid equipped with the $L_1$ metric $d((x_1,x_2),(y_1,y_2)) = |x_1-y_1|+|x_2-y_2|$. Let
\begin{displaymath}
    N = \setof{\{x,y\}\in I(n_1,n_2)\times I(n_1,n_2)}{d(x,y)=1}
\end{displaymath}
be the set of unit neighbors in $I(n_1,n_2)$. For a permutation $\pi$ of $I(n_1,n_2)$, let
\begin{displaymath}
    f(\pi) = \sum_{\{x,y\}\in N}d(\pi(x),\pi(y)).
\end{displaymath}
Then the set $\mathcal P(n_1,n_2)$ of Construction \ref{Co:Pi} is not empty, the function $f$ attains maximum at $\pi$ if and only if $\pi\in \mathcal P(n_1,n_2)$, and this maximum is equal to
\begin{equation}\label{Eq:Max}
    n_1n_2(n_1+n_2) - 4 - n_1 - w_{2n_1+2n_2-4}(n_1,n_2),
\end{equation}
where $w_k(n_1,n_2)$ is as in \eqref{Eq:WeightedBall}.
\end{theorem}
\begin{proof}
Let $n_i=2t_i$. Recall that $f(\pi)$ is maximized if and only if $g(\pi)$ is minimized, and
\begin{displaymath}
    f(\pi) = \sum_{i=1}^2\left(\sum S_i^> - \sum S_i^<\right) - g(\pi)
\end{displaymath}
by Proposition \ref{Pr:ExactWithH}. By Corollary \ref{Cr:Key}, $g(\pi)\ge 4+n_1+w(B\cup C)$. The discussion preceding Construction \ref{Co:Pi} shows that equality is achieved in Corollary \ref{Cr:Key} if and only if $\pi\in \mathcal P(n_1,n_2)$, in which case $w(B\cup C) = w_{2n_1+2n_2-4}(n_1,n_2)$ because $B\cup C$ must be a disk in $I(n_1,n_2)$. By Proposition \ref{Pr:SpecialCasesL1}, $\mathcal P(n_1,n_2)\ne\emptyset$.

It remains to show that
\begin{displaymath}
    \sum_{i=1}^2\left(\sum S_i^> - \sum S_i^<\right) = n_1n_2(n_1+n_2).
\end{displaymath}
We have
\begin{displaymath}
    \sum S_1^> - \sum S_1^< = 4n_2((t_1+1)+(t_1+2)+\cdots+(t_1+t_1)-1-2-\cdots-t_1) = 4n_2t_1^2.
\end{displaymath}
Similarly, $\sum S_2^> - \sum S_2^< = 4n_1t_2^2$. Hence
\begin{displaymath}
    \sum_{i=1}^2\left(\sum S_i^> - \sum S_i^<\right) = 4n_2t_1^2 + 4n_1t_2^2 = n_1n_2(n_1+n_2),
\end{displaymath}
and we are through.
\end{proof}

It follows from our results that in the four exceptional cases $n_1=n_2\in\{4$, $8$, $12$, $16\}$, the maximum of $f$ must be less than
\eqref{Eq:Max}. We claim that the maximum is actually equal to
\begin{equation}\label{Eq:ReducedMax}
    n_1n_2(n_1+n_2) - 6 - n_1 - w_{2n_1+2n_2-4}(n_1,n_2),
\end{equation}
that is, two less than in the generic case. It is easy to construct permutations $\pi$ of $I(n_1,n_2)$ for which \eqref{Eq:ReducedMax} is attained, by modifying Construction \ref{Co:Pi}. It is more difficult to prove that the maximum must drop by two compared to the generic case, not just by one. The details will be presented elsewhere.

\section{Comments and open problems}

\subsection{Maximum separation of neighbors for permutations on the torus}

By mimicking the proof of Lemma \ref{Lm:Grid}, we obtain:

\begin{lemma}\label{Lm:TorusGrid}
Let $n_1\le n_2$ be even integers. Consider an $n_1\times n_2$ grid on a torus with $n_1n_2/2$ red and $n_1n_2/2$ blue vertices. Then there are at least $2n_1$ red-blue edges.

Suppose that there are precisely $2n_1$ red-blue edges. If $n_1<n_2$ then the grid can be split into monochromatic connected components by two vertical cuts. If $n_1=n_2$ then the grid can be split into monochromatic connected components by two vertical cuts or by two horizontal cuts.
\end{lemma}



\begin{theorem}
Let $n_1\le n_2$ be positive even integers, $n_i=2t_i$. Let $X$ be the $n_1\times n_2$ toroidal grid, and let $d$ be the $L_1$ metric on $X$. A permutation $\pi$ of $X$ maximizes $f$ of \eqref{Eq:f} if and only if it consists of an $S_1^<\times S_2^<$, $S_1^>\times S_2^>$ checkerboard adjacent to an $S_1^<\times S_2^>$, $S_1^>\times S_2^<$ checkerboard split by two vertical lines (or two horizontal lines if $n_1=n_2$) so that either along all defect lines the row defects alternate between $t_1$, $t_1+1$, or along all defect lines the column defects alternate between $t_2$, $t_2+1$. Such permutations exist.
\end{theorem}
\begin{proof}
Since there are no corners and boundaries on $X$, $f(\pi)$ is maximized if and only if $g(\pi) = 2(\sum D_1^> + \sum D_2^> - \sum D_1^< - \sum D_2^<)$ is minimized. As usual, color homogeneous entries blue and heterogeneous entries red. Every red-blue edge corresponds to a defect. By Lemma \ref{Lm:TorusGrid}, the number of red-blue edges is at least $2n_1$, and the bound $2n_1$ is attained only by two vertical dividing lines (or by two horizontal dividing lines when $n_1=n_2$, a possibility we will discard from now on).

Additional defects are avoided if and only if the red cells form an $S_1^<\times S_2^<$, $S_1^>\times S_2^>$ checkerboard and the blue cells form an $S_1^<\times S_2^>$, $S_1^>\times S_2^<$ checkerboard. This can certainly be arranged on $X$.

Consider now the cells along one of the vertical defect lines. Suppose there is a row defect along this line. Then, as in the non-toroidal case, all defects along this line will be row defects. Moreover, we claim that all defects along the other defect line will also be row defects. Suppose we encounter the defect $(<,<)|(<,>)$ across the first defect line, in a given row. If $t_2$ is odd, we will then see $(<,<)|(<,>)$ across the second defect line, while if $t_2$ is even, we will see $(>,>)|(>,<)$ there, in the same row.

The rest is easy.
\end{proof}

\begin{example}
A permutation $\pi$ on the toroidal $6\times 6$ grid that maximizes $f(\pi)$, with two vertical defect lines:
\begin{displaymath}
    \begin{array}{|c|c|c|c|c|c|}
        \hline
        \chh(6,4)&\cll(3,1)&\clh(2,4)&\chl(6,1)&\clh(3,4)&\cll(2,1)\\
        \hline
        \cll(1,1)&\chh(5,4)&\chl(4,1)&\clh(1,4)&\chl(5,1)&\chh(4,2)\\
        \hline
        \chh(6,5)&\cll(3,2)&\clh(2,5)&\chl(6,2)&\clh(3,5)&\cll(2,2)\\
        \hline
        \cll(1,2)&\chh(5,5)&\chl(4,2)&\clh(1,5)&\chl(5,2)&\chh(4,5)\\
        \hline
        \chh(6,6)&\cll(3,3)&\clh(2,6)&\chl(6,3)&\clh(3,6)&\cll(2,3)\\
        \hline
        \cll(1,3)&\chh(5,6)&\chl(4,3)&\clh(1,6)&\chl(5,3)&\chh(4,6)\\
        \hline
    \end{array}
\end{displaymath}
\end{example}

\subsection{Open problems}

In \cite{DV}, the neighbor-separation problem was solved for permutations on $1$-dimensional intervals with respect to the $L_1$ metric:

\begin{theorem}[{{\cite[Theorem 3.1]{DV}}}]\label{Th:DV}
Let $X=\{1,\dots,n\}$ and let $\d(i,j) = |i-j|$. Then the maximum value of $f$ defined by \eqref{Eq:f} is
\begin{displaymath}
    \begin{array}{ll}
        (2t^2-1)/(2t-1),&\text{if $n=2t$, and}\\
        (2t^2+2t-1)/(2t),&\text{if $n=2t+1$.}
    \end{array}
\end{displaymath}

When $n=2t$, the maximum is attained by precisely those permutations $\pi$ that oscillate between $\{1,\dots,t\}$, $\{t+1,\dots,n\}$ and satisfy $(\pi(1),\pi(n))\in\{(t,t+1)$, $(t+1,t)\}$.

When $n=2t+1$, the maximum is attained precisely by those permutations $\pi$ that oscillate between $\{1,\dots,t\}$, $\{t+1,\dots,n\}$ and satisfy $(\pi(1),\pi(n))\in\{(t+1,t+2)$, $(t+2,t+1)\}$, and by those that oscillate between $\{1,\dots,t+1\}$, $\{t+2,\dots,n\}$ and satisfy $(\pi(1),\pi(n))\in\{(t,t+1)$, $(t+1,t)\}$.
\end{theorem}

We found the $2$-dimensional case in which at least one of the lengths is odd difficult to handle and we therefore ask:

\begin{problem}
Let $n_1$, $n_2$ be positive integers, not both even. Describe the permutations $\pi$ on $I(n_1,n_2)$ for which $f(\pi)$ is maximized with respect to the $L_1$ metric.
\end{problem}

The situation will be more intricate for $3$-dimensional intervals. When $n_1$, $n_2$, $n_3$ are even and approximately the same, $f(\pi)$ appears to be maximized by a triple checkerboard with defects occurring along two perpendicular planes through the center of the prism. However, when $n_1$ is much bigger than $n_2$ and $n_3$, $f(\pi)$ appears to be maximized by a triple checkerboard with defects occurring along two or more planes perpendicular to the first dimension.

\begin{problem}
Let $n_1$, $n_2$, $\dots$, $n_k$ be positive integers. Describe the permutations $\pi$ of the $n_1\times n_2\times\dots\times n_k$ grid that maximize $f(\pi)$ with respect to the $L_1$ metric.
\end{problem}

\section*{Acknowledgment}

We thank Dan Daly for pointing out several relevant references. We also thank an anonymous referees for very careful reading of the manuscript and for several suggestions that improved the presentation.

\end{document}